\documentclass[12pt,reqno]{amsart}

\usepackage{mathrsfs,amsfonts,amsthm,amsmath,amssymb,thmtools}

\usepackage{enumitem}                
\usepackage{lineno}                 

\usepackage[hyperindex, pdfencoding = auto, psdextra, bookmarksdepth = 4]{hyperref}  
\hypersetup{
    bookmarksopen      = true,
    bookmarksopenlevel = 1,
    bookmarksnumbered  = true,
    pdftitle    = {},
    pdfauthor   = {},
    linktoc     = page,
    anchorcolor = green,
    colorlinks  = true,        
    linkcolor   = black,
    citecolor   = blue,
    filecolor   = blue,
    urlcolor    = magenta,
}

\usepackage[nameinlink]{cleveref}    

\theoremstyle{plain}

\newtheorem{theorem}{Theorem}[section]
\newtheorem{lemma}[theorem]{Lemma}
\newtheorem{proposition}[theorem]{Proposition}
\newtheorem{corollary}[theorem]{Corollary}
\newtheorem{question}[theorem]{Question}

\theoremstyle{definition}   
\newtheorem{definition}[theorem]{Definition}

\makeatletter
    \def\subsection{\@startsection{subsection}{2}%
    \z@{.5\linespacing\@plus.7\linespacing}{.3\linespacing}%
    {\normalfont\bfseries}}
\makeatother 

\makeatletter                                         
    \newcommand{\LeftEqNo}{\let\veqno\@@leqno}        
\makeatother

\numberwithin{equation}{section}                      

\begin{document}

\

\vspace{-2cm}

\title[Irreducible operators in von Neumann algebras]{Irreducible operators in von Neumann algebras}

\author{Sukitha Adappa}
\address{Sukitha Adappa, School of Mathematical and Statistical Sciences, Arizona State University, Tempe, 85281, US}
\email{Sukitha.Adappa@asu.edu}

\author{Minghui Ma}
\address{Minghui Ma, School of Mathematical Sciences, Dalian University of Technology, Dalian, 116024, China}
\email{minghuima@dlut.edu.cn}

\author{Junhao Shen}
\address{Junhao Shen, Department of Mathematics \& Statistics, University of New Hampshire, Durham, 03824, US}
\email{junhao.shen@unh.edu}

\author{Rui Shi}
\address{Rui Shi, School of Mathematical Sciences, Dalian University of Technology, Dalian, 116024, China}
\email{ruishi@dlut.edu.cn}
\thanks{Rui Shi, Minhua Ma, and Shanshan Yang are partly supported by NSFC (No.12271074), the Fundamental Research Funds for the Central Universities (No.DUT23LAB305). Minghui Ma is also supported by the Postdoctoral Fellowship Program of CPSF (No. GZC20252022).}

\author{Shanshan Yang}
\address{Shanshan Yang, School of Mathematical Sciences, Dalian University of Technology, Dalian, 116024, China}
\email{yss@mail.dlut.edu.cn}



\begin{abstract}
Let $\mathcal{M}$ be a separable von Neumann algebra with center $\mathcal{Z}(\mathcal{M})$.
An operator $T$ in $\mathcal{M}$ is called irreducible if the von Neumann algebra $W^*(T)$ generated by $T$ has trivial relative commutant, i.e., $W^*(T)'\cap\mathcal{M}=\mathcal{Z}(\mathcal{M})$.
In this paper, we show that irreducible operators in $\mathcal{M}$ form a norm-dense $G_\delta$ set, which is a generalization of Halmos' theorem.
Moreover, we prove that every operator in $\mathcal{M}$ is the sum of two irreducible operators, which is an analogue of Radjavi's theorem.
\end{abstract}

\maketitle

\section{Introduction}

Let $\mathcal{B}(\mathcal{H})$ denote the algebra of all bounded linear operators on a complex Hilbert space $\mathcal{H}$.
A \emph{von Neumann algebra} is a self-adjoint unital subalgebra of $\mathcal{B}(\mathcal{H})$ that is closed in the weak-operator topology.
Every von Neumann algebra can be decomposed into the direct sum
of those of type $\mathrm{I}_n$, type $\mathrm{I}_\infty$, type $\mathrm{II}_1$, type $\mathrm{II}_\infty$, and type $\mathrm{III}$ by the type decomposition theorem for von Neumann algebras \cite[Theorem 6.5.2]{KR2}.
A von Neumann algebra is called \emph{separable} if it has a separable predual space.
A \emph{factor} is a von Neumann algebra whose center consists of scalar multiples of the identity.
By definition, $\mathcal{B}(\mathcal{H})$ is a factor of type $\mathrm{I}$.

Let $\mathcal{M}$ be a factor with identity $I$.
Recall that an operator $T$ is said to be \emph{irreducible} in $\mathcal{M}$ if $PT=TP$ implies that either $P=0$ or $P=I$ for each projection $P$ in $\mathcal{M}$, i.e., $W^*(T)'\cap\mathcal{M}=\mathbb{C}I$, where $W^*(T)$ denotes the von Neumann algebra generated by $T$.
Otherwise, $T$ is called \emph{reducible} in $\mathcal{M}$.
In 1968, P. Halmos \cite{Hal} proved that the set of irreducible operators on a separable complex Hilbert space $\mathcal{H}$ is a norm-dense $G_\delta$ subset of $\mathcal{B}(\mathcal{H})$.
Later, H. Radjavi and P. Rosenthal \cite{RR} provided a very short proof of Halmos' theorem.
In 1976, D. Voiculescu \cite{Voi} proved a noncommutative Weyl-von Neumann theorem and obtained the norm-density of reducible operators in $\mathcal{B}(\mathcal{H})$.
Recently, D. Hadwin et al. \cite{HMS} established the norm-density of reducible operators in separable properly infinite factors.
Without the assumption of separability, J. Shen and R. Shi \cite[Theorem 6.6]{SS19} proved that the set of reducible operators in every non-$\Gamma$ type $\mathrm{II}_1$ factor $\mathcal{M}$ is nowhere norm-dense.
In particular, the set of irreducible operators in $\mathcal{M}$ is norm-dense.
As an example, let $\mathbb{F}_S$ be the free group generated by an uncountable set $S$.
Then the group factor $\mathcal{L}(\mathbb{F}_S)$ is neither separable nor singly generated, but the set of irreducible operators in $\mathcal{L}(\mathbb{F}_S)$ is norm-dense.
On the other hand, there exists a large class of nonseparable type $\mathrm{II}_1$ factors containing no irreducible operator \cite[Example 5.3]{SS19}.
By replacing $W^*(T)'\cap\mathcal{M}=\mathbb{C}I$ with $W^*(T)'\cap\mathcal{M}=\mathcal{Z}(\mathcal{M})$ in the definition of irreducible operators in factors, it is natural to generalize the irreducibility of operators in von Neumann algebras as follows.

\begin{definition}\label{def-IR-opt}
Let $\mathcal{M}$ be a von Neumann algebra with center $\mathcal{Z}(\mathcal{M})$.
A von Neumann subalgebra $\mathcal{P}$ of $\mathcal{M}$ is said to have \emph{trivial relative commutant} if the equality $\mathcal{P}^{\prime} \cap \mathcal{M} = \mathcal{Z}(\mathcal{M})$ holds.
An operator $T$ in $\mathcal{M}$ is \emph{irreducible} if the von Neumann algebra $W^*(T)$ generated by $T$ has trivial relative commutant, i.e., $W^*(T)'\cap\mathcal{M}=\mathcal{Z}(\mathcal{M})$.
\end{definition}

A. Sinclair and R. Smith \cite[Theorem 8]{SS} proved that every separable type $\mathrm{II}_1$ von Neumann algebra contains an injective von Neumann subalgebra which has trivial relative commutant.
By matrix construction techniques, we show in \Cref{prop P+iB} that there is an irreducible operator whose real part is a projection in each separable von Neumann algebra.
Recall that an operator $T$ in a von Neumann algebra $\mathcal{M}$ is called a \emph{generator} if $\mathcal{M}=W^*(T)$.
Clearly, every generator of $\mathcal{M}$ is irreducible, and it is an open question whether every separable von Neumann algebra is singly generated.
The authors of \cite{WFS} proved the norm-density of irreducible operators in separable factors, which is a generalization of Halmos' theorem.
A refined version in semifinite factors can be found in \cite{Shi}.
In this paper, we mainly investigate two properties of irreducible operators in separable von Neumann algebras with nontrivial center. To this point, we need to develop new methods. As a generalization of Halmos' theorem into von Neumann algebras, we prove the following result.

\begin{theorem}\label{thm norm-dense}
    Let $\mathcal{M}$ be a separable von Neumann algebra and $\mathcal{I}$ a weak-operator dense norm-closed ideal in $\mathcal{M}$.
    Then for any $T\in\mathcal{M}$ and $\varepsilon>0$, there exists an operator $K\in\mathcal{I}$ with $\|K\|<\varepsilon$ such that $T+K$ is irreducible in $\mathcal{M}$.
\end{theorem}

Let $\mathcal{M}=\mathcal{B}(\mathcal{H})$ and $\mathcal{I}=\mathcal{K}(\mathcal{H})$, where $\mathcal{H}$ is a separable infinite-dimensional complex Hilbert space.
By \Cref{thm norm-dense}, for any $T\in\mathcal{B}(\mathcal{H})$ and $\varepsilon>0$, there exists a compact operator $K\in\mathcal{K}(\mathcal{H})$ with $\|K\|<\varepsilon$ such that $T+K$ is irreducible in $\mathcal{B}(\mathcal{H})$.
Moreover, as a consequence of \Cref{thm norm-dense}, the set of irreducible operators in every separable von Neumann algebra is a norm-dense $G_\delta$ set (see \Cref{cor norm-dense}).
Inspired by \cite[Definition 6.1]{SS19} and \Cref{thm norm-dense}, we propose the following question.

\begin{question}
    Let $\mathcal{N}$ be a von Neumann algebra, $\mathcal{M}$ a von Neumann subalgebra of $\mathcal{N}$, and
    \begin{equation*}
      \mathrm{IR}(\mathcal{M},\mathcal{N})=\{T\in\mathcal{M}\colon W^*(T)'\cap\mathcal{N}=\mathcal{M}'\cap\mathcal{N}\}.
    \end{equation*}
    Suppose that $\mathcal{M}$ is separable.
    Is the set $\mathrm{IR}(\mathcal{M},\mathcal{N})$ norm-dense in $\mathcal{M}$?
\end{question}

Note that the above question is true for $\mathcal{M}=\mathcal{N}$ by \Cref{thm norm-dense}, and we will write $\operatorname{IR}(\mathcal{M})$ instead of $\operatorname{IR}(\mathcal{M},\mathcal{M})$.
Moreover, if $\mathcal{N}=\mathcal{B}(\mathcal{H})$, then $\mathrm{IR}(\mathcal{M},\mathcal{B}(\mathcal{H}))$ becomes $\operatorname{Gen}(\mathcal{M})$, the set of all generators of $\mathcal{M}$, by the double commutant theorem \cite[Theorem 5.3.1]{KR1}. In general, $\operatorname{Gen}(\mathcal{M})\subseteq \mathrm{IR}(\mathcal{M},\mathcal{N}) \subseteq \mathrm{IR}(\mathcal{M})$.
The question is quite interesting when $\mathcal{N}$ is a type $\mathrm{II}_1$ factor and $\mathcal{M}$ is an irreducible subfactor with finite index, i.e., $\mathcal{M}'\cap\mathcal{N}=\mathbb{C}I$ and $[\mathcal{N}\colon\mathcal{M}]<\infty$.

P. Fillmore and D. Topping \cite{FT} showed that every operator in $\mathcal{B}(\mathcal{H})$ is the sum of four irreducible operators, where $\mathcal{H}$ is a separable infinite-dimensional complex Hilbert space.
H. Radjavi \cite{Rad} improved this result by proving that every operator in $\mathcal{B}(\mathcal{H})$ is the sum of two irreducible operators.
With a different proof, the authors of \cite{SS20} proved that every operator in a separable factor is the sum of two irreducible operators.
By developing the techniques with respect to the centers of von Neumann algebras, we establish an analogue in von Neumann algebras as follows.

\begin{theorem}\label{thm sum-irreducible}
    Let $\mathcal{M}$ be a separable von Neumann algebra.
    Then every operator in $\mathcal{M}$ is the sum of two irreducible operators in $\mathcal{M}$.
\end{theorem}

Before closing the introduction, we raise two more problems which we hope will contribute to a deeper understanding of irreducible operators in von Neumann algebras.
Inspired by \Cref{def-IR-opt}, we introduce \emph{strongly irreducible operators} in von Neumann algebras, which are a  natural analogue of Jordan blocks in $M_n(\mathbb{C})$.

\begin{definition}\label{def-SIR-opt}
An operator $T$ in a von Neumann algebra $\mathcal{M}$ is said to be \emph{strongly irreducible} if $XTX^{-1}$ is irreducible in $\mathcal{M}$ for every invertible operator $X$ in $\mathcal{M}$.
Denote by $\operatorname{SIR}(\mathcal{M})$ the set of strongly irreducible operators in $\mathcal{M}$.
\end{definition}

Note that the spectrum of each strongly irreducible operator is connected in factors.
For $\mathcal{B}(\mathcal{H})$, we write $\operatorname{SIR}(\mathcal{H})$ instead of $\operatorname{SIR}(\mathcal{B}(\mathcal{H}))$.
D. Herrero and C. Jiang proved in \cite[Main theorem]{HJ-1990} that the norm-closure of $\operatorname{SIR}(\mathcal{H})$ is the set of all operators whose spectra are connected.
Also, from \cite[Theorem 8.2]{FJLX-2013}, it yields that $\operatorname{SIR}(\mathcal{R})$ is nonempty for the hyperfinite type $\mathrm{II}_1$ factor $\mathcal{R}$.

\begin{question}  \label{question-01}
Are there strongly irreducible operators in each factor $\mathcal{M}$?
If so, is it true that
    \begin{equation*}
        \overline{\operatorname{SIR}(\mathcal{M})}^{\Vert\cdot\Vert} = \{T\in \mathcal{M}\colon\sigma(T)~\text{is connected}\} ?
    \end{equation*}
\end{question}

In comparison with \cite{Rad}, it follows from \cite[Theorem 3.1.1]{JW-2006} that every operator in $\mathcal{B}(\mathcal{H})$ is the sum of two strongly irreducible operators.
In light of this, it is natural to raise the following question.

\begin{question}\label{question-02}
Is every operator in a factor $\mathcal{M}$ the sum of two strongly irreducible operators?
Equivalently, is the equality $\mathcal{M}=\operatorname{SIR}(\mathcal{M})+\operatorname{SIR}(\mathcal{M})$ true?
\end{question}

\section{Proofs}

We start by presenting some basic results that will be used later.
The proofs of \Cref{thm norm-dense} and \Cref{thm sum-irreducible} will be given in \Cref{subsec norm-dense} and \Cref{subsec sum-irreducible}, respectively.

\subsection{Preliminaries}\label{subsec preliminary}

Recall that the \emph{central support} $C_T$ of an operator $T$ in a von Neumann algebra $\mathcal{M}$ is the minimal central projection with the property $C_TT=T$.
Clearly, $C_{ZT}=ZC_T$ for every central projection $Z$ in $\mathcal{M}$ by definition.
In the next two lemmas, we perturb an operator $T\in P\mathcal{M}Q$ to be $Y$ such that $C_Y=C_PC_Q$.
We will use \Cref{lem CY} to construct the operators $Y_{jk}$ in the proof of \Cref{lem diagonal}.

\begin{lemma}\label{lem CV}
Let $P$ and $Q$ be projections in a von Neumann algebra $\mathcal{M}$.
Then there exists a partial isometry $V$ in $P\mathcal{M}Q$ such that $C_V=C_PC_Q$.
\end{lemma}

\begin{proof}
By the comparison theorem \cite[Theorem 6.2.7]{KR2}, there exists a central projection $Z$ in $\mathcal{M}$ such that $ZP\precsim ZQ$ and $(I-Z)P\succsim(I-Z)Q$.
Let $V_1$ and $V_2$ be partial isometries in $P\mathcal{M}Q$ such that
\begin{equation*}
  V_1V_1^*=ZP,\quad V_1^*V_1\leqslant ZQ,\quad V_2^*V_2=(I-Z)Q, \quad V_2V_2^*\leqslant (I-Z)P.
\end{equation*}
Let $V=V_1+V_2\in P\mathcal{M}Q$.
Then $C_V\leqslant C_PC_Q$.
Since $V_1\in\mathcal{M}Z$ and $V_2\in\mathcal{M}(I-Z)$, we have
\begin{equation*}
  C_V=C_{V_1}+C_{V_2}=ZC_P+(I-Z)C_Q.
\end{equation*}
It follows that $C_VC_PC_Q=C_PC_Q$.
Thus, $C_V=C_PC_Q$.
\end{proof}

\begin{lemma}\label{lem CY}
Let $P$ and $Q$ be projections in a von Neumann algebra $\mathcal{M}$.
Then for any operator $T\in P\mathcal{M}Q$ and $\varepsilon>0$, there exists an operator $Y\in P\mathcal{M}Q$ such that
\begin{equation*}
  \|T-Y\|<\varepsilon\quad\text{and}\quad C_Y=C_PC_Q.
\end{equation*}
\end{lemma}

\begin{proof}
Let $P_1=(I-C_T)P$ and $Q_1=(I-C_T)Q$.
By \Cref{lem CV}, there exists a partial isometry $V$ in $P_1\mathcal{M}Q_1$ such that $C_V=C_{P_1}C_{Q_1}=(I-C_T)C_PC_Q$.
Let $Y=T+\frac{\varepsilon}{2}V\in P\mathcal{M}Q$.
Then $C_Y\leqslant C_PC_Q$ and
\begin{equation*}
  C_Y=C_T+C_V=C_T+(I-C_T)C_PC_Q.
\end{equation*}
It follows that $C_YC_PC_Q=C_PC_Q$.
Therefore, $Y$ has the desired property.
\end{proof}

The following lemma characterizes irreducible operators in the direct sum of a family of von Neumann algebras.

\begin{lemma}\label{lem direct-sum}
Suppose that $\mathcal{M}=\bigoplus_{j\in\Lambda}\mathcal{M}_j$, where $\{\mathcal{M}_j\}_{j\in\Lambda}$ is a family of von Neumann algebras.
Let $T=\bigoplus_{j\in\Lambda}T_j\in\mathcal{M}$.
Then $T$ is irreducible in $\mathcal{M}$ if and only if $T_j$ is irreducible in $\mathcal{M}_j$ for each $j\in\Lambda$.
\end{lemma}

\begin{proof}
Let $P=\bigoplus_{j\in\Lambda}P_j$ be a projection in $\mathcal{M}$.
Then $PT=TP$ if and only if $P_jT_j=T_jP_j$ for each $j\in\Lambda$.
Hence
\begin{equation*}
  W^*(T)'\cap\mathcal{M}=\bigoplus_{j\in\Lambda}W^*(T_j)'\cap\mathcal{M}_j.
\end{equation*}
Note that $\mathcal{Z}(\mathcal{M})=\bigoplus_{j\in\Lambda}\mathcal{Z}(\mathcal{M}_j)$.
This completes the proof.
\end{proof}

Recall that $\mathcal{Z}(E\mathcal{M}E)=\mathcal{Z}(\mathcal{M})E$ for every projection $E$ in $\mathcal{M}$ by \cite[Proposition 5.5.6]{KR1}.
We provide a useful criterion for the irreducibility of operators as follows.

\begin{lemma}\label{lem ETE-irreducible}
Let $A$ and $B$ be self-adjoint operators in a von Neumann algebra $\mathcal{M}$.
Suppose that $\{E_j\}_{j\in\Lambda}$ is a family of projections in $W^*(A)$ such that
\begin{enumerate}
\item [$(1)$] $I=\sum_{j\in\Lambda}E_j$;

\item [$(2)$] $C_{E_jBE_k}=C_{E_j}C_{E_k}$ for all $j\ne k$;

\item [$(3)$] $E_j(A+iB)E_j$ is irreducible in $E_j\mathcal{M}E_j$ for each $j\in\Lambda$.
\end{enumerate}
Then $A+iB$ is irreducible in $\mathcal{M}$.
\end{lemma}

\begin{proof}
Let $P$ be a projection in $\mathcal{M}$ that commutes with $A+iB$.
Since $E_j\in W^*(A)$, we have $PE_j=E_jP$ for each $j\in\Lambda$.
It follows that $PE_j$ commutes with $E_j(A+iB)E_j$ and hence the condition $(3)$ implies that
\begin{equation*}
  PE_j\in\mathcal{Z}(E_j\mathcal{M}E_j)=\mathcal{Z}(\mathcal{M})E_j.
\end{equation*}
In particular, $PE_j=Z_jE_j$ for some $Z_j\in\mathcal{Z}(\mathcal{M})$.
Since $PB=BP$, it is clear that
\begin{equation*}
  Z_jE_jBE_k=E_jPBE_k=E_jBPE_k=E_jBE_kZ_k=Z_kE_jBE_k,
\end{equation*}
i.e., $(Z_j-Z_k)E_jBE_k=0$.
By the condition $(2)$, we have
\begin{equation}  \label{eq-ZCE}
    (Z_j-Z_k)C_{E_j}C_{E_k}=0\quad\text{for all}~j\ne k.
\end{equation}
Note that \eqref{eq-ZCE} holds trivially for $j=k$.
Let $T$ be an operator in $\mathcal{M}$.
Then
\begin{equation*}
  (Z_j-Z_k)E_jTE_k=(Z_j-Z_k)C_{E_j}C_{E_k}E_jTE_k=0
  \quad\text{for all}~j,k\in\Lambda.
\end{equation*}
It follows that $PE_jTE_k=Z_jE_jTE_k=E_jTE_kZ_k=E_jTE_kP.$
Thus, the condition $(1)$ yields $PT=TP$.
Therefore, $P\in\mathcal{Z}(\mathcal{M})$.
\end{proof}

If $\mathcal{N}$ is a von Neumann algebra with an irreducible operator and $n\geqslant 2$, then there exists an irreducible operator in $M_n(\mathbb{C})\otimes\mathcal{N}$ with a nice property.

\begin{lemma}\label{lem P+iB}
Suppose that $\mathcal{N}$ is a von Neumann algebra, $n\geqslant 2$, and $\mathcal{M}=M_n(\mathbb{C})\otimes\mathcal{N}$.
If there exists an irreducible operator in $\mathcal{N}$, then there exists an irreducible operator $P+iB$ in $\mathcal{M}$, where $P$ is a projection and $B$ is a self-adjoint operator.
\end{lemma}

\begin{proof}
Let $a+ib$ be an irreducible operator in $\mathcal{N}$, where $a$ and $b$ are self-adjoint.
Without loss of generality, we may assume that $\frac{3}{5}I\leqslant a,b\leqslant\frac{4}{5}I$.
Let $\{e_{jk}\}_{j,k=1}^n$ be a system of matrix units in $M_n(\mathbb{C})$ and we define two operators in $\mathcal{M}$ by
\begin{equation*}
  P=
  \begin{pmatrix}
    I-a & \sqrt{\frac{a-a^2}{n-1}} & \cdots & \sqrt{\frac{a-a^2}{n-1}} \\
    \sqrt{\frac{a-a^2}{n-1}} & \frac{a}{n-1} & \cdots & \frac{a}{n-1} \\
    \vdots & \vdots & \ddots & \vdots \\
    \sqrt{\frac{a-a^2}{n-1}} & \frac{a}{n-1} & \cdots & \frac{a}{n-1}
  \end{pmatrix}\quad\text{and}\quad
  B=
  \begin{pmatrix}
    b & & & &\\
    & \frac{1}{2}I & & & \\
    & & \frac{1}{3}I & & \\
    & & & \ddots & \\
    & & & & \frac{1}{n}I
  \end{pmatrix}.
\end{equation*}
It is clear that $P$ is a projection and $B$ is self-adjoint.
By Borel function calculus, we have $e_{jj}\otimes I\in W^*(B)$ for $1\leqslant j\leqslant n$.
By considering $(e_{jj}\otimes I)P(e_{j+1,j+1}\otimes I)$, the operators
\begin{equation*}
  e_{12}\otimes\sqrt{\frac{a-a^2}{n-1}}\quad\text{and}\quad e_{j,j+1}\otimes\frac{a}{n-1}\quad (2\leqslant j\leqslant n-1)
\end{equation*}
are all in $W^*(P+iB)$.
Since $\frac{3}{5}I\leqslant a\leqslant\frac{4}{5}I$, both $\sqrt{\frac{a-a^2}{n-1}}$ and $\frac{a}{n-1}$ are invertible operators in $\mathcal{N}$.
It follows that $e_{j,j+1}\otimes I\in W^*(P+iB)$ for $1\leqslant j\leqslant n-1$ and hence
\begin{equation*}
  \{e_{jk}\otimes I\}_{j,k=1}^n\subseteq W^*(P+iB).
\end{equation*}
From this, we obtain that $W^*(P+iB)=M_n(\mathbb{C})\otimes W^*(a+ib)$.
Therefore, $P+iB$ is an irreducible operator in $\mathcal{M}$.
\end{proof}

Suppose that $\mathcal{P}$ is a separable hyperfinite type $\mathrm{II}_1$ von Neumann algebra.
Then $\mathcal{P}$ can be decomposed as a tensor product $\mathcal{P}\cong\mathcal{R}\,\overline{\otimes}\,\mathcal{A}$ by \cite{Con}, where $\mathcal{R}$ is the hyperfinite type $\mathrm{II}_1$ factor and $\mathcal{A}$ is an abelian von Neumann algebra.
By \cite[Corollary 5.15]{She}, we assume that $\mathcal{R}$ is generated by two self-adjoint operators $A$ and $B$.
By \cite[Theorem 7.12]{RR73}, the abelian von Neumann algebra $W^*(\{A\otimes I\}\cup(I\otimes\mathcal{A}))$ is generated by a self-adjoint operator $A_1$.
Let $B_1=B\otimes I$.
It is clear that $\mathcal{P}$ is generated by the operator $A_1+iB_1$.
This fact will be directly applied in the proof of the following proposition, which is a comparable version of the lemma in \cite{FT}.

\begin{proposition}\label{prop P+iB}
Let $\mathcal{M}$ be a separable von Neumann algebra.
Then there exists an irreducible operator $P+iB$ in $\mathcal{M}$, where $P$ is a projection and $B$ is a self-adjoint operator.
\end{proposition}

\begin{proof}
According to \cite[Theorem 7.12]{RR73}, every separable abelian von Neumann algebra is generated by a self-adjoint operator.
If $\mathcal{M}$ is abelian, then we can choose $P+iB$ with $P=0$ and $W^*(B)=\mathcal{M}$.
If $\mathcal{M}$ is of type $\mathrm{I}_n$ for some integer $n\geqslant 2$, then we can write that $\mathcal{M}\cong M_n(\mathbb{C})\otimes\mathcal{Z}(\mathcal{M})$ by \cite[Theorem 6.6.5]{KR2} and the conclusion follows from \Cref{lem P+iB}.

If $\mathcal{M}$ is either properly infinite or of type $\mathrm{II}_1$, then $\mathcal{M}\cong M_2(\mathbb{C})\otimes\mathcal{N}$ for some von Neumann algebra $\mathcal{N}$ by \cite[Lemma 6.3.3, Lemma 6.5.6]{KR2}.
If $\mathcal{N}$ is properly infinite, then $\mathcal{N}$ has a generator by \cite{Wog}, which is clearly irreducible in $\mathcal{N}$.
If $\mathcal{N}$ is of type $\mathrm{II}_1$, then there exists a hyperfinite type $\mathrm{II}_1$ von Neumann subalgebra $\mathcal{P}$ of $\mathcal{N}$ such that $\mathcal{P}'\cap\mathcal{N}=\mathcal{Z}(\mathcal{N})$ by \cite[Theorem 8]{SS}.
Hence every generator of $\mathcal{P}$ is an irreducible operator in $\mathcal{N}$ and we complete the proof by \Cref{lem P+iB}.
The general case follows from \Cref{lem direct-sum} and the type decomposition theorem 
\cite[Theorem 6.5.2]{KR2}.
\end{proof}

\subsection{Proof of \texorpdfstring{\Cref{thm norm-dense}}{theorem ref}}\label{subsec norm-dense}

Let $\mathcal{I}$ be a (two-sided) norm-closed ideal in a von Neumann algebra $\mathcal{M}$.
We show that certain operators in $\mathcal{M}$ can be perturbed to be irreducible operators with respect to $\mathcal{I}$ in the following lemma.

\begin{lemma}\label{lem diagonal}
Let $\mathcal{M}$ be a separable von Neumann algebra and $\mathcal{I}$ a norm-closed ideal in $\mathcal{M}$.
Suppose that $\{E_j\}_{j=1}^{\infty}$ is a sequence of projections in $\mathcal{I}$ such that $I=\sum_{j=1}^{\infty}E_j$.
Let $A$ and $B$ be self-adjoint operators in $\mathcal{M}$ with $A=\sum_{j=1}^{\infty}\alpha_jE_j$, where $\{\alpha_j\}_{j=1}^{\infty}$ is a sequence of real numbers.
Then for any $\varepsilon>0$, there exists an operator $K\in\mathcal{I}$ with $\|K\|<\varepsilon$ such that $(A+iB)+K$ is irreducible in $\mathcal{M}$.
\end{lemma}

\begin{proof}
By spectral theory, for each $j\geqslant 1$, there are projections $\{F_{jk}\}_{k=1}^{n_j}$ and real numbers $\{\lambda_{jk}\}_{k=1}^{n_j}$ such that
\begin{equation*}
  E_j=\sum_{k=1}^{n_j}F_{jk}\quad\text{and}\quad
  \left\|E_jBE_j-\sum_{k=1}^{n_j}\lambda_{jk}F_{jk}\right\|
  <\frac{\varepsilon}{2^{j+1}}.
\end{equation*}
Let $T_j=E_jBE_j-\sum_{k=1}^{n_j}\lambda_{jk}F_{jk}\in E_j\mathcal{M}E_j\subseteq\mathcal{I}$ and $T=\sum_{j=1}^{\infty}T_j$.
It is clear that $T\in\mathcal{I}$ and $\|T\|<\frac{\varepsilon}{2}$.
Let $B_0=B-T$.
Then $F_{jk}B_0F_{jk}=\lambda_{jk}F_{jk}$.

For the sake of simplicity, we rename $\{F_{jk}\}_{1\leqslant k\leqslant n_j,j\geqslant 1}$ as $\{F_j\}_{j=1}^{\infty}$.
Without loss of generality, we assume that $A=\sum_{j=1}^{\infty}\alpha_jF_j$
and $F_jB_0F_j=\lambda_jF_j$ for each $j\geqslant 1$.
By \Cref{prop P+iB}, there exists an irreducible operator $P_j+iH_j$ in $F_j\mathcal{M}F_j$, where $P_j$ is a projection and $H_j$ is a self-adjoint operator with $\|H_j\|<\frac{\varepsilon}{2^{j+3}}$ for each $j\geqslant 1$.
There are two sequences $\{\beta_j\}_{j=1}^\infty$ of real numbers and $\{\varepsilon_j\}_{j=1}^\infty$ of positive numbers such that
\begin{equation*}
  |\alpha_j-\beta_j|<\frac{\varepsilon}{2^{j+3}},\quad 0<\varepsilon_j<\frac{\varepsilon}{2^{j+3}}\quad\text{for all}~j\geqslant 1,
\end{equation*}
and $\{\beta_j,\beta_j+\varepsilon_j\}\cap\{\beta_k,\beta_k+\varepsilon_k\}=\emptyset$ for all $j\ne k$.
By \Cref{lem CY}, there exists an operator $Y_{jk}\in F_j\mathcal{M}F_k$ such that $\|F_jB_0F_k-Y_{jk}\|<\frac{\varepsilon}{2^{j+k+3}}$ and $C_{Y_{jk}}=C_{F_j}C_{F_k}$ for all $j<k$.
Let $Y_{jk}=Y_{kj}^*$ for all $j>k$.
Now we define
\begin{equation*}
  A_1=\sum_{j=1}^{\infty}(\beta_jF_j+\varepsilon_jP_j)\quad\text{and}\quad
  B_1=\sum_{j=1}^{\infty}(\lambda_jF_j+H_j)+\sum_{j\ne k}Y_{jk}.
\end{equation*}
Then $A_1-A, B_1-B_0\in\mathcal{I}$ and $\|A_1-A\|,\|B_1-B_0\|<\frac{\varepsilon}{4}$.
It follows that
\begin{equation*}
  (A_1+iB_1)-(A+iB)\in\mathcal{I}\quad\text{and}\quad
  \|(A_1+iB_1)-(A+iB)\|<\varepsilon.
\end{equation*}
By Borel function calculus, we have $F_j\in W^*(A_1)$ for each $j\geqslant 1$.
By the definition of $P_j+iH_j$, the operator $F_j(A_1+iB_1)F_j$ is irreducible in $F_j\mathcal{M}F_j$ for each $j\geqslant 1$.
Moreover, the central support of $F_jB_1F_k=Y_{jk}$ is $C_{F_j}C_{F_k}$ for all $j\ne k$.
Therefore, $A_1+iB_1$ is irreducible in $\mathcal{M}$ by \Cref{lem ETE-irreducible}.
\end{proof}

We present the proof of \Cref{thm norm-dense} by using \cite[Theorem 4.8]{AP} and \Cref{lem diagonal}.

\begin{proof}[Proof of Theorem {\rm\ref{thm norm-dense}}]
By \cite[Theorem 4.8]{AP}, there exists a self-adjoint operator $K_1\in\mathcal{I}$ with $\|K_1\|<\frac{\varepsilon}{2}$ such that $\mathrm{Re}(T)+K_1=\sum_{j=1}^{\infty}\alpha_jE_j$, where $\{E_j\}_{j=1}^{\infty}$ are projections in $\mathcal{I}$ with sum $I$ and $\{\alpha_j\}_{j=1}^{\infty}$ are real numbers.
By \Cref{lem diagonal}, there exists an operator $K_2\in\mathcal{I}$ with $\|K_2\|<\frac{\varepsilon}{2}$ such that $T+K_1+K_2$ is irreducible in $\mathcal{M}$.
We complete the proof by taking $K=K_1+K_2$.
\end{proof}

The following corollary is a direct consequence of \Cref{thm norm-dense}.

\begin{corollary}\label{cor norm-dense}
The set of irreducible operators in every separable von Neumann algebra is a norm-dense $G_\delta$ set.
\end{corollary}

\begin{proof}
Let $\mathcal{M}$ be a separable von Neumann algebra.
Then the set of irreducible operators in $\mathcal{M}$ is norm-dense by \Cref{thm norm-dense}.
Let
\begin{equation*}
  \mathcal{P}=\{A\in\mathcal{M}\colon 0\leqslant A\leqslant I\}
  \quad\text{and}\quad
  \mathcal{P}_0=\mathcal{P}\setminus\mathcal{Z}(\mathcal{M}).
\end{equation*}
Since $\mathcal{P}$ is weak-operator compact and $\mathcal{Z}(\mathcal{M})$ is weak-operator closed, $\mathcal{P}_0$ is locally weak-operator compact.
The remaining part of the proof is a slight modification of that of \cite[Theorem 2.1]{WFS}.
\end{proof}

\subsection{Proof of \texorpdfstring{\Cref{thm sum-irreducible}}{Thm-cref}}\label{subsec sum-irreducible}
We present a standard result in the theory of von Neumann algebras as follows.

\begin{lemma}\label{lem P<I-P}
Let $P$ be a projection in a von Neumann algebra $\mathcal{M}$ such that $C_P=I$ and $P\precsim I-P$.
Then there exists a family of projections $\{P_j\}_{j\in\Lambda}$ and an index $j_0\in\Lambda$ such that
\begin{equation*}
  I-P=\sum_{j\in\Lambda}P_j\quad\text{and}\quad P_j\precsim P_{j_0}\sim P\quad\text{for all}~j\in\Lambda.
\end{equation*}
\end{lemma}

\begin{proof}
Since $P\precsim I-P$, there exists a projection $Q$ in $\mathcal{M}$ such that $P\sim Q\leqslant I-P$.
Let $\{P_j\}_{j\in\Lambda_0}$ be a maximal orthogonal family of nonzero subprojections of $I-P-Q$ such that $P_j\precsim P$ for each $j\in\Lambda_0$.
Then $I-P-Q=\sum_{j\in\Lambda_0}P_j$.
Otherwise, let
\begin{equation*}
  Q_0=I-P-Q-\sum_{j\in\Lambda_0}P_j\ne 0.
\end{equation*}
Then $C_{Q_0}C_P=C_{Q_0}\ne 0$.
By \cite[Proposition 6.1.8]{KR2}, there exists a nonzero projection $P_0\leqslant Q_0$ such that $P_0\precsim P$.
That contradicts the maximality of $\{P_j\}_{j\in\Lambda_0}$.
Let $P_{j_0}=Q$ and $\Lambda=\{j_0\}\cup\Lambda_0$.
This completes the proof.
\end{proof}

The special case $\mathcal{N}=W^*(T)'\cap\mathcal{M}$ of the following lemma will be used in the proof of \Cref{thm sum-irreducible}.

\begin{lemma}\label{lem completely-reducible}
    Let $\mathcal{N}$ be a von Neumann subalgebra of $\mathcal{M}$ with $\mathcal{Z}(\mathcal{M})\subseteq\mathcal{N}\subseteq\mathcal{M}$.
    Suppose that $\mathcal{N}Z\ne\mathcal{Z}(\mathcal{M})Z$ for every nonzero central projection $Z$ in $\mathcal{M}$.
    Then there exists a projection $P$ in $\mathcal{N}$ such that $C_P=C_{I-P}=I$.
    Furthermore, $P$ can be chosen such that $P\precsim I-P$.
\end{lemma}

\begin{proof}
    Let $\{P_j\}_{j\in\Lambda}$ be a maximal family of nonzero projections in $\mathcal{N}$ with mutually orthogonal central supports $\{C_{P_j}\}_{j\in\Lambda}$ such that
    \begin{equation*}
      C_{P_j}=C_{C_{P_j}-P_j}\quad\text{for each}~j\in\Lambda.
    \end{equation*}
    We claim that $I=\sum_{j\in\Lambda}C_{P_j}$.
    Otherwise, let $Z=I-\sum_{j\in\Lambda}C_{P_j}\ne 0$.
    By assumption, there exists a nonzero projection $Q\in\mathcal{N}Z\setminus\mathcal{Z}(\mathcal{M})Z$.
    It follows that $Q\ne C_QZ$ and hence $C_Q(Z-Q)\ne 0$.
    Let $Z_0=C_QC_{Z-Q}$ and $P_0=Z_0Q$.
    Then $Z_0$ is a nonzero subprojection of $Z$ and
    \begin{equation*}
      C_{P_0}=Z_0C_Q=Z_0\quad\text{and}\quad C_{Z_0-P_0}=C_{Z_0(Z-Q)}=Z_0C_{Z-Q}=Z_0.
    \end{equation*}
    That contradicts the maximality of $\{P_j\}_{j\in\Lambda}$ when we consider $\{P_0\}\cup\{P_j\}_{j\in\Lambda}$.

    Let $P=\sum_{j\in\Lambda}P_j$.
    Then $C_P=C_{I-P}=I$.
    By the comparison theorem \cite[Theorem 6.2.7]{KR2}, there exists a central projection $Z$ in $\mathcal{M}$ such that
    \begin{equation*}
      ZP\precsim Z(I-P)\quad\text{and}\quad (I-Z)P\succsim(I-Z)(I-P).
    \end{equation*}
    Let $P'=ZP+(I-Z)(I-P)$.
    Then $P'$ has the desired property $P'\precsim I-P'$.
\end{proof}

Now we are ready to provide the proof of \Cref{thm sum-irreducible}.

\begin{proof}[Proof of Theorem {\rm\ref{thm sum-irreducible}}]
    Let $T=A+iB$ be an operator in $\mathcal{M}$, where $A$ and $B$ are self-adjoint.
    We will construct irreducible operators $T_1$ and $T_2$ in $\mathcal{M}$ such that $T=T_1+T_2$.

    Let $\{Z_j\}_{j\in\Lambda}$ be a maximal orthogonal family of central projections in $\mathcal{M}$ such that $TZ_j$ is irreducible in $\mathcal{M}Z_j$ for each $j\in\Lambda$.
    Let $Z_0=\sum_{j\in\Lambda}Z_j$.
    Then $\mathcal{M}=\mathcal{M}Z_0\oplus\mathcal{M}(I-Z_0)$ and $TZ_0$ is irreducible in $\mathcal{M}Z_0$ by \Cref{lem direct-sum}.
    In the von Neumann algebra $\mathcal{M}Z_0$, we can take $T_1=T_2=\frac{1}{2}TZ_0$.
    Thus, it suffices to consider the operator $T(I-Z_0)$ in $\mathcal{M}(I-Z_0)$.

    Without loss of generality, we may assume that $Z_0=0$.
    Then for any nonzero central projection $Z$ in $\mathcal{M}$, we have $\big(W^*(T)'\cap\mathcal{M}\big)Z\ne\mathcal{Z}(\mathcal{M})Z$.
    By \Cref{lem completely-reducible}, there exists a projection $P$ in $W^*(T)'\cap\mathcal{M}$ such that $C_P=C_{I-P}=I$ and $P\precsim I-P$.
    By \Cref{lem P<I-P} and the separability of $\mathcal{M}$, there exists a sequence $\{P_j\}_{j=1}^{\infty}$ of projections in $\mathcal{M}$ such that
    \begin{equation*}
      I-P=\sum_{j=1}^{\infty}P_j\quad\text{and}\quad P_j\precsim P_1\sim P
      \quad\text{for all}~j\geqslant 1.
    \end{equation*}
    Let $\{V_j\}_{j=1}^{\infty}$ be a sequence of partial isometries in $\mathcal{M}$ such that
    \begin{equation*}
      V_j^*V_j=P_j\quad\text{and}\quad V_jV_j^*\leqslant V_1V_1^*=P\quad\text{for all}~j\geqslant 1.
    \end{equation*}
    By \Cref{cor norm-dense}, there exists an irreducible operator $A_0+iB_0$ in $P\mathcal{M}P$, where $A_0$ and $B_0$ are positive invertible operators in $P\mathcal{M}P$.
    Let $Q$ be a projection in $\mathcal{M}$ that commutes with $\{A_0,B_0\}\cup\{V_j\}_{j=1}^{\infty}$.
    Then $QP=PQ$ commutes with $A_0+iB_0$.
    It follows that
    \begin{equation*}
      QP\in\mathcal{Z}(P\mathcal{M}P)=\mathcal{Z}(\mathcal{M})P
    \end{equation*}
    and we can write $QP=ZP$ for some $Z\in\mathcal{Z}(\mathcal{M})$.
    Since
    \begin{equation*}
      QP_j=QV_j^*PV_j=V_j^*QPV_j=V_j^*ZPV_j=ZP_j\quad\text{for all}~j\geqslant 1,
    \end{equation*}
    we have $Q=QP+\sum_{j=1}^{\infty}QP_j=Z\in\mathcal{Z}(\mathcal{M})$.
    Hence $W^*(\{A_0,B_0\}\cup\{V_j\}_{j=1}^{\infty})$ has trivial relative commutant, i.e.,
    \begin{equation}\label{equ A-B-V}
      W^*(\{A_0,B_0\}\cup\{V_j\}_{j=1}^{\infty})'\cap\mathcal{M}
      =\mathcal{Z}(\mathcal{M}).
    \end{equation}
    Let $\{\alpha_j\}_{j=1}^{\infty}$ be a bounded sequence of distinct positive numbers such that
    \begin{equation*}
      \alpha_j>1+\|A_0\|+\|AP\|+\|A(I-P)\|\quad\text{for all}~j\geqslant 1.
    \end{equation*}
    Define
    \begin{equation*}
      X=A_0+\sum_{j=1}^{\infty}\alpha_jP_j\quad\text{and}\quad Y=X+\sum_{j=1}^{\infty}\frac{1}{2^j}(B_0V_j+V_j^*B_0).
    \end{equation*}
    Let $T_1=A+X+iY$ and $T_2=-X+i(B-Y)$.
    It is clear that $T=T_1+T_2$.
    We claim that both $T_1$ and $T_2$ are irreducible in $\mathcal{M}$.

    Note that $P$ commutes with $T=A+iB$, $P\in W^*(A+X)$ by Borel function calculus, and
    \begin{equation*}
      X=PYP+(I-P)Y(I-P)\in W^*(T_1).
    \end{equation*}
    Then $A_0,P_j\in W^*(X)\subseteq W^*(T_1)$ for each $j\geqslant 1$.
    Since $B_0V_1=2P(Y-X)P_1\in W^*(T_1)$ and $B_0$ is a positive invertible operator in $P\mathcal{M}P$, we have $B_0,V_1\in W^*(T_1)$ by polar decomposition.
    Since $B_0V_j=2^jP(Y-X)P_j\in W^*(T_1)$, we have $V_j\in W^*(T_1)$ for each $j\geqslant 1$.
    Similarly, since $P(Y-B)P_j=B_0V_j$, we have $A_0,B_0,V_j\in W^*(T_2)$ for each $j\geqslant 1$.
    Therefore, $T_1$ and $T_2$ are irreducible in $\mathcal{M}$ by \eqref{equ A-B-V}.
\end{proof}


\end{document}